\documentclass[11pt]{article}
\usepackage[top=1in,bottom=1in,right=1in,left=1in]{geometry}
\usepackage{amssymb,amsmath,amsthm}
\usepackage{arydshln}
\usepackage{verbatim}
\usepackage{titlesec}
\usepackage{subfigure}
\usepackage[linesnumbered,ruled]{algorithm2e}
\usepackage{bm}
\usepackage{pdfpages}
\usepackage{wrapfig}

\definecolor{carminepink}{rgb}{0.92, 0.3, 0.26}
\definecolor{burntsienna}{rgb}{0.91, 0.45, 0.32}

\definecolor{forestgreen}{rgb}{0.13, 0.55, 0.13}
\usepackage[colorlinks, linkcolor=blue,citecolor=forestgreen,urlcolor = black, bookmarks=true]{hyperref}

\usepackage[T1]{fontenc}
 
\usepackage[nameinlink, noabbrev, capitalize]{cleveref}
\usepackage{times}

\allowdisplaybreaks

\crefname{equation}{}{}
\crefrangeformat{equation}{(#3#1#4) --~(#5#2#6)}
\crefmultiformat{equation}{(#2#1#3)}{, (#2#1#3)}{, (#2#1#3)}{, (#2#1#3)}
\crefname{lem}{Lemma}{Lemmas}
\crefname{section}{Section}{Sections}
\crefname{subsubsubsection}{Section}{Sections}
\crefname{rem}{Remark}{Remarks}
\crefname{cor}{Corollary}{Corollaries}
\crefname{figure}{Figure}{Figures}
\crefname{table}{Table}{Tables}
\Crefname{lem}{Lemma}{Lemmas}
\Crefname{line}{Line}{Lines}
\Crefname{fact}{Fact}{Facts}
\crefname{thm}{Theorem}{Theorems}
\crefname{def}{Definition}{Definitions}
\crefname{assumption}{Assumption}{Assumptions}

\usepackage{commath}

\newtheorem{thm}{Theorem}[section]

\newtheorem{defn}[thm]{Definition}
\newtheorem{lem}[thm]{Lemma}

\newtheorem{cor}[thm]{Corollary}

\newtheorem{fact}[thm]{Fact}

\theoremstyle{remark}
\newtheorem{rmk}{Remark}[section]

\title{On Algorithmic Robustness of Corrupted Markov Chains
}
\author{Jason Gaitonde\\
Duke University\\
\texttt{jason.gaitonde@duke.edu}
\and Elchanan Mossel\\
Massachusetts Institute of Technology\\
\texttt{elmos@mit.edu}}
\date{}

\begin{document}
\maketitle
\begin{abstract}
We study the algorithmic robustness of  
general finite Markov chains in terms of their stationary distributions to general, adversarial corruptions of the transition matrix.  
We show that for Markov chains admitting a spectral gap, variants of the \emph{PageRank}
chain are robust in the sense that, given an \emph{arbitrary} corruption of the edges emanating from an $\epsilon$-measure of the nodes, the PageRank distribution of the corrupted chain will be $\mathsf{poly}(\varepsilon)$ close in total variation to the 
original distribution under mild conditions on the restart distribution. 
Our work thus shows that PageRank serves as a simple regularizer against broad, realistic corruptions with algorithmic guarantees that are dimension-free and scale gracefully in terms of necessary and natural parameters. 
\end{abstract}
\newpage

\section{Introduction}

Markov chains are fundamental and ubiquitous models of stochastic processes, with profound real-world applications across biology, computer science, economics, probability theory, and statistics, among other scientific domains. 
Markov chains provide a mathematically rich and convenient modeling abstraction of random processes in complex systems that satisfy a natural 
property: the random transition to the next state depends only on the current state, and not on the entire trajectory. Given a state space $\mathcal{X}$, the entire evolution of a Markov chain can thus be compactly described by a row-stochastic matrix $\mathsf{P}\in \mathbb{R}^{\mathcal{X}\times \mathcal{X}}$, where $\mathsf{P}_{ij}$ denotes the probability of transitioning to state $j$ given that the system is at state $i$.

A central object of interest in the study of Markov chains is the \emph{stationary distribution}, which (under standard conditions) is the unique distribution $\pi$ that remains invariant upon transitioning according to the chain. Equivalently, $\pi$ is the unique distribution on $\mathcal{X}$ such that $\pi \mathsf{P}=\pi$, where we write the distribution as a row vector. A fundamental result in the theory of Markov chains is that under mild conditions the stationary distribution also describes the long-run behavior of a stochastic system whose trajectory evolves according to the underlying transitions of the Markov process. As a result, the stationary distribution will also capture important statistical information about the macroscopic properties of the system from its microscopic dynamics, for instance as a measure of centrality or importance of states, equilibrium in game-theoretic or physical systems, and long-run behavior in important operations settings like queueing and ride-sharing.

While any Markov chain is associated to its stationary distribution, the stationary distribution can often be the more important statistical object in its own right. A modern perspective introduced by Metropolis, et al.~\cite{metropolis1953equation} was to \emph{design} Markov chains such that the desired distribution $\pi$ is stationary, so that one may sample using the chain. This approach has been applied extensively in statistics and engineering with notable examples including the Metropolis algorithm~\cite{metropolis1953equation}, Gibbs sampling ~\cite{geman1984stochastic}, and the PageRank algorithm~\cite{brin1998anatomy}. 
Understanding the mixing times of such algorithms has been a major area of research in probability and theoretical computer science (see~\cite{montenegro2006mathematical,levin2017markov} for textbook treatments). 

How robust is the stationary distribution to \emph{natural} forms of corruption in the underlying Markov chain? A key motivating example arises in the context of directed networks of Internet websites. Here, each state in the Markov chain represents a website, and there is a directed (possibly weighted) edge from one site to another if the first links to the second. A natural Markov model for such a network is a variant of the random walk, where a user navigates from one site to another by randomly following hyperlinks. Brin and Page~\cite{brin1998anatomy} and Kleinberg~\cite{DBLP:journals/jacm/Kleinberg99} hypothesized a remarkable connection between random walks and site importance: a website should be considered important if it is linked to by other important websites and thus if it is visited more often. This correspondence naturally leads to defining centrality in terms of the stationary distribution, a perspective that has since been widely adopted across many other diverse domains (see e.g. the survey of~\cite{DBLP:journals/siamrev/Gleich15} on applications of PageRank across the natural and social sciences).

In many of the applications above, states (e.g. agents or websites) may have significant incentives to artificially 
modify the 
stationary distribution by altering the transitions from sites that they can control. This behavior has been extensively seen in real-world actors aiming to increase visibility in Google search~\cite{genius,parasite}. An important feature is that corruptions of transitions from such states can be \emph{essentially arbitrary}. 
Other perturbations of chains can be the result of small innocent computational or statistical errors. 
For example, transitions from states may be construed using statistical models 
and sampling or modeling errors in these statistical models may result in minor differences in a large proportion of all outgoing transitions which could compound globally. In realistic applications where the essential algorithmic task is determining the original stationary distribution, it is desired to be tolerant to both kinds of corruptions.

In this work, we consider this algorithmic problem of recovering the stationary distribution of a Markov chain under practically motivated corruptions of the transition matrix:

\begin{quote}
\centering
    \emph{Under what conditions and forms of adversarial corruptions can the stationary distribution of be a Markov chain be algorithmically recovered given the corrupted transition matrix?}
\end{quote}

Of course, given the wide-ranging applications of Markov chains, questions of this form have been studied for several decades, dating back to at least Schweitzer~\cite{schweitzer1968perturbation}. But to the best of our knowledge, the question of what conditions provably enable algorithmic recovery for \emph{arbitrary corruptions} of the form above has not been previously investigated. We discuss important prior results, which predominantly take a matrix perturbation approach, as well as their limitations in more detail in \Cref{sec:related}; roughly speaking, these results typically require that \emph{each} state's transitions are perturbed by at most $o(1)$. Therefore, these approaches do not furnish a robust recovery algorithm when corruptions of outgoing transitions are arbitrary for 
a small fraction of the nodes, as is the case for 
adversarial states in real-world settings mentioned above. 

Our main results are thus to identify simple conditions that enable robust recovery of the stationary condition under a large class of adversarial corruptions. In particular, we allow an adversary to arbitrarily corrupt an $\varepsilon$-fraction of the transitions, as measured by the unknown stationary distribution over directed edges (see \Cref{def:corruption} for a precise definition). 
Crucially, this formulation is general enough to allow for completely arbitrary corruptions on a $\varepsilon$-fraction of the state space with respect to the \emph{unknown} stationary distribution $\pi$, which can be arbitrarily far from the stationary distribution of the corrupted chain.

To illustrate a simple, but already powerful special case of our general results, we show the following:

\begin{thm}[informal, \Cref{cor:spread}]
\label{thm:informal_1}
    Suppose that $\mathsf{P}$ is a Markov chain on $n$ states with spectral gap $\gamma > 0$ and with stationary distribution $\pi$ such that $n\pi(x)\in [\alpha,\alpha^{-1}]$ for all $x$ and some $\alpha > 0$.
    Then there is an efficient algorithm that, given a transition matrix $\widetilde{\mathsf{P}}$ of $\mathsf{P}$ where any $\varepsilon$-fraction of states have \emph{arbitrarily} corrupted transitions, computes $\hat{\pi}$ 
    such that the total variation between
    $\hat{\pi}$ and $\pi$ is 
    $\widetilde{O}\left(\frac{1}{\alpha}\sqrt{\frac{\varepsilon}{\gamma}}\right)$.
\end{thm}

In fact, we show that the algorithm achieving this is simply an instantiation of the seminal PageRank distribution of Brin and Page~\cite{brin1998anatomy}. In general, the same guarantee can be achieved under mild smoothness conditions that we identify on the restart distribution used in the algorithm. In the setting of dynamic web networks, where the stationary distribution is updated on a regular basis, 
this assumption will be quite natural and could be achieved simply by using a previously computed stationary distribution. 

\subsection{Main Results and Techniques}
\label{sec:overview}

We now describe our results in more detail; we defer to  \Cref{sec:prelims} for all formal definitions used in this section. The main focus of our work is to algorithmically recover the stationary distribution under a new, important class of corruptions of a Markov chain:

\begin{defn}[Corrupted Markov Chains]
\label{def:corruption}
    Let $\mathcal{X}$ be a finite state space and let $\mathsf{P}$ be a Markov transition matrix on $\mathcal{X}$ with stationary measure $\pi$. A transition matrix $\widetilde{\mathsf{P}}$ is a \textbf{$\varepsilon$-corruption} of $\mathsf{P}$ if it holds that
    \begin{equation*}
        \sum_{x,y\in \mathcal{X}}\pi(x)\vert \mathsf{P}(x,y)-\widetilde{\mathsf{P}}(x,y)\vert\leq \varepsilon.
    \end{equation*}
    We write $\widetilde{\pi}$ for a stationary distribution of the corrupted chain.
\end{defn}
We remark that this definition is not symmetric since the corruption is measured with respect to the stationary distribution of the original chain $\pi$, which we stress is unknown. This choice is well-motivated in at least three respects: first, it is easy to find examples where corrupting even a single site can yield otherwise indistinguishable chains that share all other natural properties (see \Cref{rmk:pi_neccessary}) but where the new stationary distribution is far, so one must weight the corruptions by the original importance of the site to obtain any guarantees. Second, it is natural to expect that it is more costly for an adversary to manipulate transitions at important sites in the original chain, so \Cref{def:corruption} can also be viewed as a budget constraint. Finally, it is reasonable to expect that a central system administrator is more capable of observing macroscopic changes at large states compared to small states, making larger corruptions harder to implement. 

Importantly, this formulation unifies and broadly extends two important settings where we may hope to approximately recover the stationary distribution:
\begin{itemize}
    \item First, consider the setting where \emph{each} site has transitions that are corrupted by at most $\varepsilon$ in total variation. This is the standard setting captured by all prior work that we are aware of and naturally models the robustness of Markov chains to small sampling errors in each step. In this case, it is immediate to see that this is captured by \Cref{def:corruption} since we have a uniform bound of the difference at each site, independent of the actual stationary distribution.
    \item On the other extreme, consider the setting where an adversary controls a subset $\mathcal{C}\subseteq \mathcal{X}$ of states such that $\pi(\mathcal{C})\leq \varepsilon/2$. In this case, the adversary may set \emph{completely arbitrary transitions} out of these states to produce a $\varepsilon$-corruption according to \Cref{def:corruption}. In important settings where $\pi(x)=\Theta(1/n)$ for all states $x\in \mathcal{X}$, this corresponds to a robust statistics framework where the adversary controls a $\Theta(\varepsilon)$ fraction of rows of the matrix $\mathsf{P}$. Again, note that in this case, the corrupted stationary distribution can be \emph{arbitrarily far} from the uncorrupted distribution since even a single arbitrarily corrupted state can be absorbing.
\end{itemize}

We now formally define the PageRank chain:
\begin{defn} \label{def:page_rank}
Given a Markov transition matrix 
$\mathsf{P}$, distribution $\mu$ on the states and $0 \leq \delta \leq 1$, the 
\emph{PageRank chain} with parameter $\delta$ and restart distribution $\mu$
is defined by:
\[
\mathsf{P}(\delta):=(1-\delta)\mathsf{P}+\delta \bm{1}^T\mu.
\]
\end{defn}

The starting point for our results is that prior work has already established the robustness of chains to small site-wise perturbations under sufficiently fast mixing. 
In particular, it is well-known (c.f. \Cref{cor:pr_close}, see e.g. Mitrophanov~\cite{mitrophanov2003stability} and Vial and Subramanian~\cite{subramanian}) 
that if $\mathsf{P}$ has a \emph{spectral gap} $\gamma$, then for any choice of $\delta$ and restart distribution $\mu$, the corresponding PageRank stationary distribution $\pi_{\delta}$ satisfies a total variation distance bound of
\begin{equation*}
    d_{\mathsf{TV}}(\pi,\pi_{\delta})\lesssim \frac{\delta}{\gamma}.
\end{equation*}

Our goal is to 
obtain similar guarantees for the more realistic corruptions permitted under \Cref{def:corruption}. Our key insight is that we can show that 
under the mild condition that the restart distribution $\mu$ satisfies a bound of $\|\mu/\pi\|_{p,\pi}\leq \beta$ for some choice of $p>1$ (see \Cref{sec:prelims} for the precise definition of the norm), the PageRank version of the \emph{corrupted chain} will be close to the PageRank distribution of the \emph{uncorrupted chain}. If we write the corresponding PageRank matrix of the corruption as $\widetilde{\mathsf{P}}(\delta)$ with stationary distribution $\widetilde{\pi}_{\delta}$, then we can show (\Cref{cor:corrupted_close}) that:
\begin{equation*}
    d_{\mathsf{TV}}(\pi_{\delta},\widetilde{\pi}_{\delta})\lesssim \frac{\beta\varepsilon^{1/q}}{\delta},
\end{equation*}
where $q\geq 1$ is the dual exponent of $p$ satisfying $1/p+1/q=1$. This step crucially relies on the restart interpretation and contractiveness of Markov operators to translate from the guarantee that $\widetilde{\mathsf{P}}$ is a $\varepsilon$-corruption measured with respect to the unknown distribution $\pi$.

At this point, we may use the triangle inequality and balance terms by setting $\delta$ appropriately to establish our main result:
\begin{thm}
\label{thm:intro_main}
     Let $\mathsf{P}$ be a Markov chain with stationary distribution $\pi$ and spectral gap at least $\gamma$, and suppose that $\widetilde{\mathsf{P}}$ is a $\varepsilon$-corruption. Given the corrupted chain $\widetilde{\mathsf{P}}$ and a restart distribution $\mu$ that is $(\beta,p)$-smooth with respect to $\pi$, then PageRank with restart distribution $\mu$ and suitable choice of $\delta$ outputs a distribution $\widehat{\pi}$ such that
    \begin{equation*}
        d_{\mathsf{TV}}(\widehat{\pi},\pi)=\widetilde{O}\left(\sqrt{\frac{\beta\varepsilon^{1/q}}{\gamma}}\right).
    \end{equation*}
\end{thm}

\Cref{thm:informal_1} is a simple corollary in the important special case under the assumption that all states have stationary measure of order $\Theta(1/n)$, showing that one obtains nontrivial recovery guarantees when when an arbitrary $\varepsilon\lesssim \gamma$ fraction of states are arbitrarily corrupted by simply using PageRank with a uniform restart (see \Cref{cor:spread} for the formal guarantee). However, \Cref{thm:intro_main} allows for significantly milder conditions on the restart distribution that nonetheless yield algorithmic guarantees.

\subsection{Related Work}
\label{sec:related}

Conceptually, \Cref{thm:informal_1} is related to the types of corruptions studied in algorithmic robust statistics (see~\cite{diakonikolas2023algorithmic} for a textbook treatment). In problems like robust mean estimation or supervised learning, it is typically assumed that a small constant fraction of the observed data can be \emph{arbitrarily corrupted}. As in our setting, these forms of corruptions preclude the use of standard, direct algorithms that can be extremely sensitive to large outliers, and thus require new insights. A key distinction is that underlying samples in these problems can be viewed as statistically independent in a very strong sense: if the corrupted samples were somehow known, one could just run a standard algorithm on the inliers. By contrast, the stationary distribution of a Markov chain is an inherently global object that depends intimately on the relationships between the transitions between all pairs of states.

\textbf{Matrix Perturbation Bounds.} As mentioned before, there is a vast literature that investigates the robustness of Markov chains. To obtain an explicit bound on the distance between stationary measures, all works we are aware of employ a \emph{matrix perturbation approach} of the following form: if $\widetilde{P}=\mathsf{P}+\Delta$ for some error matrix $\Delta$ expressing the difference, then one hopes to show that
\begin{equation}
\label{eq:perturbation}
    \|\pi-\widetilde{\pi}\|\leq \kappa(\mathsf{P})\cdot\|\Delta\|,
\end{equation}
for some suitable choices of norm and some measure $\kappa(\cdot)$ that measures the robustness of the chain. Natural choices of norm for the left-hand side include the entrywise or relative maximum difference as well as the total variation distance. Consequently, these works aim to address a weaker question: when do perturbations have a provably benign effect on the \emph{new} stationary distribution?

The central conclusion of these lines of work is that, for restricted classes of corruptions, the distance between the stationary distributions can be determined by ratio of the amount of corruption and some measure of the mixing time. The intuition is that if the original chain mixes faster than it encounters a perturbation, then the corrupted chain will not be largely affected. As a prototypical example that we will also employ as a preliminary estimate (c.f. \Cref{cor:pr_close}), several authors (e.g. Mitrophanov~\cite{mitrophanov2003stability}, Hunter~\cite{HUNTER2005217,hunter2006mixing}, and Vial and Subramanian~\cite{subramanian}) show that 
\begin{equation*}
    d_{\mathsf{TV}}(\pi,\widetilde{\pi})\lesssim \frac{\|\Delta\|}{\gamma},
\end{equation*} 
where $\gamma$ is the spectral gap of the chain and $\|\Delta\|$ is the maximum $\ell_1$ norm of a row in the error matrix. Note however that this norm precisely measures the total corruption allowed at any one site, which is insufficient to capture the more realistic settings we wish to control. 

In principle, the new stationary distribution can be written in terms of probabilistic quantities associated to the original chain and the perturbation like hitting times; see e.g. the classical results of Schweitzer~\cite{schweitzer1968perturbation}, Seneta~\cite{seneta1993sensitivity}, Laserre~\cite{laserre}, and Hunter~\cite{HUNTER2005217} among many others. Variations on this matrix perturbation approach that depend on specific features of the chain or using other metrics have also appeared in abundance, as in~\cite{meyer,thiede, 10.3150/17-BEJ938} among many other works. We defer to the surveys~\cite{CHO2001137,abbas2016critical,rudolf2024perturbations} for a more comprehensive comparison of this vast literature. 

We stress that to obtain any nontrivial recovery under the notion of corruption of \Cref{def:corruption}, \emph{the matrix perturbation approach is inherently insufficient}: the left-hand side of \Cref{eq:perturbation} simply will not generally be small in these settings since the corrupted stationary distribution will be far even for trivial perturbations that, for instance, create absorbing states. The major contribution of our work is showing that more careful algorithms can still succeed in robust recovery in these cases.

\textbf{PageRank.} The PageRank algorithm was introduced by Brin and Page~\cite{brin1998anatomy} as a ``random surfer'' model of the web: at each step, the surfer chooses a random hyperlink at the current stage with probability $(1-\delta)$ (where $\delta$ is sometimes called the ``damping factor") and otherwise randomly restarts at a random site sampled according to some fixed \emph{restart distribution} $\mu$. A typical choice would just be a uniform site, and this alteration ensures irreducibility (and so convergence) of the chain. In our work, we view PageRank as a \emph{regularizer} that can be used for algorithmic recovery under adversarial corruptions under mild conditions on the restart.

Various works have established a connection between the mixing properties of the chain and PageRank. Caputo and Quattropani~\cite{caputo} study the behavior of PageRank in the setting of the random digraph model and establish a fascinating trichotomy behavior on the total variation distribution between the original and PageRank stationary distribution depending on the scaling between the restart probability and mixing time of the Markov chain. These ideas were studied more generally by Vial and Subramanian~\cite{subramanian}, who establish a similar trichotomy under \emph{cutoff}. Their work gives upper bounds on the total variation distance of a Markov chain and a corrupted chain that may perturb each transition by $\alpha$ in total variation when the ratio between the mixing time and $\alpha$ approach $c\in [0,\infty]$. In the case that $c=0$, they show that the total variation distance in fact tends to $0$. They also show that under reversibility and cutoff, these bounds are asymptotically sharp and attained precisely by suitable PageRank perturbations. These works provide elementary, but powerful tools analyzing the tradeoff between mixing and robustness and we rely on key elements of their analysis to treat the more general class of corruptions allowed by \Cref{def:corruption}.

\textbf{Corrupted Dynamics.} The setting where each state has $\varepsilon$-corrupted transitions can be motivated from the perspective of robust sampling. The recent work of Chin, Mossel, Moitra, and Sandon~\cite{adversary} explores the power of a adversary that controls a $\varepsilon$-subset of \emph{sites} in spin systems on graphs. This setting can essentially be cast in the corruption model we consider with the extra restriction that the same set of sites is corrupted at each state of the spin system. However, their main focus is understanding how such an adversary can \emph{efficiently} affect important macroscopic properties of the chain rather than the more stringent total variation distance. We later discuss future directions aimed at unifying the results here with the kinds of phenomena considered for that class of dynamics. Recent work of Lin, Liu, and Smith~\cite{lin2025perturbation} shows that under certain technical conditions, sampling from graphical models can be robust to larger perturbations than those suggested by the mixing time heuristic above. Craiu, et al.~\cite{craiu} investigate the tightness of Markov chains when arbitrary adversarial behavior may occur, in a time-dependent way, on a bounded subset of states; note that this question is typically only meaningful in infinite state spaces.

\textbf{Acknowledgments.} This research was conducted while J.G. was at the Department of Mathematics at MIT supported by Vannevar Bush Faculty Fellowship ONR-N00014-20-1-2826 and Simons Investigator Award 622132. E.M. is supported in part by Vannevar Bush Faculty Fellowship ONR-N00014-20-1-2826, Simons Investigator Award 622132, Simons-NSF DMS-2031883, NSF Award CCF 1918421, and ONR MURI Grant N000142412742.

\section{Preliminaries}
\label{sec:prelims}
In this section, we recall the following standard definitions and results on Markov chains. Recall that a matrix $\mathsf{P}\in \mathbb{R}^{n\times n}$ is a Markov chain on $n$ states if the entries are nonnegative and the rows sum to one. The entry $\mathsf{P}_{ij}$ denotes the probability of transitioning to state $j$ when at state $i$. We will write distributions on $[n]$ as a row vector. A distribution $\pi$ is \emph{stationary} with respect to $\mathsf{P}$ if $\pi \mathsf{P}=\pi$; under standard irreducibility and aperiodicity conditions (which will always hold under our later assumptions), it is well-known that $\lim_{t\to \infty} \bm{p}\mathsf{P}^t=\pi$ for any starting distribution $\bm{p}$. For a subset $S\subseteq [n]$ we will write $\pi(S)$ to denote $\sum_{s\in S} \pi(s)$.

If $\pi$ is a distribution on $\mathcal{X}$, then we may consider the Hilbert space $L^2(\mathcal{X},\pi)$ where for any functions $f,g:\mathcal{X}\to \mathbb{R}$, we define the associated inner product
\begin{equation*}
    \langle f,g\rangle_{\pi}=\mathbb{E}_{X\sim \pi}[f(X)g(X)] = \sum_{x\in \mathcal{X}}\pi(x)f(x)g(x).
\end{equation*}
We also have the corresponding $L^p$ spaces defined in the usual way with
\begin{equation*}
\|f\|_{p,\pi}=\left(\mathbb{E}_{X\sim \pi}[\vert f(X)\vert^p]\right)^{1/p}.
\end{equation*}
When we omit $\pi$ in the subscripts, we mean the usual inner product and $\ell_p$ norms of vectors. 

The corresponding induced matrix norm $\|\mathsf{M}\|_{p\to p}$ is defined as $\sup_{f:\|f\|_{p,\pi}\leq 1} \|\mathsf{M}f\|_p$. The following contractive property of Markov operators is well-known:
\begin{fact}
\label{fact:contract}
    $\pi$ is the stationary distribution of $\mathsf{P}$, then $\|\mathsf{P}\|_{p\to p}=1$ for all $p\geq 1$.
\end{fact}
\begin{proof}
We may simply compute:
\begin{align*}
    \|\mathsf{P}f\|_{p,\pi} &= \left(\mathbb{E}_{X\sim\pi}[\vert \mathsf{P}f(X)\vert^p]\right)^{1/p}\\
    &=\left(\mathbb{E}_{X\sim\pi}[\vert \mathbb{E}_{X'\sim \mathsf{P}(X,\cdot)}[ f(X')]\vert^p\right)^{1/p}\\
    &\leq \left(\mathbb{E}_{X'\sim \pi}[\vert f(X')\vert^p]\right)^{1/p}=\|f\|_{p,\pi},
\end{align*}
by Jensen's inequality and using the fact $X'$ is also distributed as $\pi$ by stationarity.
\end{proof}

Our main goal is to recover the stationary distribution with small error in \emph{total variation distance}, which is defined between a distribution $\bm{p}$ and a distribution $\pi$ via
\begin{equation*}
    d_{\mathsf{TV}}(\bm{p},\pi):=\frac{1}{2}\|\bm{p}-\pi\|_1=\frac{1}{2}\left\|\bm{1}-\frac{\bm{p}}{\pi}\right\|_{1,\pi},
\end{equation*}
where $\bm{1}$ denotes the all-ones row vector and the ratio is considered pointwise under the convention that $0\cdot (a/0) = a$ for $a\geq 0$ in this expression. 

For our algorithmic guarantees, we will require the following definition quantifying the extent to which a warm start is not too spiked with respect to the unknown $\pi$:
\begin{defn}[Smooth Distributions]
    Let $\pi$ be a distribution on a state space $\mathcal{X}$. A distribution $\mu$ is $(\beta,p)$-smooth with respect to $\pi$ if 
    \begin{equation*}
        \left\|\frac{\mu}{\pi}\right\|_{p,\pi}\leq \beta.
    \end{equation*}
\end{defn}
In words, the $(\beta,p)$-smoothness of a distribution $\mu$ with respect to $\pi$ is a measure of how much $\mu$ overestimates $\pi$. Note that for any $p\leq p'$, $(\beta,p')$-smoothness implies $(\beta,p)$-smoothness by Jensen's inequality. The case where $p=\infty$ amounts to $\mu$ providing relative estimates for all states, while $p=1$ is trivially always satisfied with $\beta=1$. Note though that even for $p=\infty$, the total variation distance between $\pi$ and a $(\beta,p)$-smooth distribution may be quite close to $1$. We will nonetheless provide algorithmic guarantees under any fixed choice of $p>1$.

\begin{defn}[Adjoints]
    Suppose $\pi$ is a distribution on $\mathcal{X}$ that is positive. For any matrix $\mathsf{M}\in \mathbb{R}^{\mathcal{X}\times \mathcal{X}}$, the \textbf{adjoint} with respect to $\pi$ is the matrix $\mathsf{M}^*$ defined via
    \begin{equation*}
        \mathsf{M}^*(x,y) = \frac{\pi(y)\mathsf{M}(y,x)}{\pi(x)}
    \end{equation*}
\end{defn}

It is a standard fact that if $\pi$ is the unique stationary distribution for a Markov transition matrix $\mathsf{P}$, then the adjoint $\mathsf{P}^*$ is also a Markov transition matrix whose stationary distribution is $\pi$ and therefore also satisfies the conclusion of \Cref{fact:contract}. Indeed,
\begin{equation*}
    \sum_{y} \mathsf{M}^*(x,y) = \frac{1}{\pi(x)}\sum_{y} \pi(y)\mathsf{M}(y,x) = \frac{\pi(x)}{\pi(x)} = 1,
\end{equation*}
by the stationarity of $\pi$ for $\mathsf{M}$, and also
\begin{equation*}
    \sum_{x}\pi(x)\mathsf{M}^*(x,y) = \sum_{x} \pi(y)M(y,x) = \pi(y).
\end{equation*}

The benefit of the adjoint is that it is well-known to characterize the effect of transitioning according to the chain:
\begin{fact}
\label{fact:adj_dist}
    For any distribution $\nu$ on $\mathcal{X}$, define the column vector $f=(\frac{\mu}{\pi}-\bm{1})^T$. Then $\mathbb{E}_{\pi}[f]=0$ and 
    \begin{equation*}
        \mathsf{P}^*f = \left(\frac{\mu\mathsf{P}}{\pi}-1\right)^T.
    \end{equation*}
\end{fact}
\begin{proof}
    The first identity holds because $\mu$ is a distribution:
    \begin{equation*}
        \mathbb{E}_{\pi}[f(X)] = \sum_{x}\pi(x)\left(\frac{\mu(x)}{\pi(x)}-1\right)=\sum_{x}[\mu(x)-\pi(x)] = 0.
    \end{equation*}

    The second identity is standard and follows from a quick calculation:
    \begin{align*}
        \mathsf{P}^*f(x) &= \sum_{y}\mathsf{P}^*(x,y)f(y)\\
        &=\sum_{y}\frac{\pi(y)\mathsf{P}(y,x)}{\pi(x)}\left(\frac{\mu(y)}{\pi(y)}-1\right)\\
        &=\sum_{y}\left(\frac{\mu(y)\mathsf{P}(y,x)}{\pi(x)}\right) -1\\
        &=\frac{\mu \mathsf{P}}{\pi}(x) - 1.
    \end{align*}
    
\end{proof}

As discussed above, and as we will see in simple examples later, the ability to recover the stationary distribution is connected to the \emph{spectral gap} of the chain:
\begin{defn}[Spectral Gap]
    Let $\mathcal{X}$ be a finite state space and let $\mathsf{P}$ be a Markov transition matrix on $\mathcal{X}$. Then the \textbf{spectral gap} of $\mathsf{P}$ is defined as 
    \begin{equation*}
        \gamma(\mathsf{P})=1-\sup_{f\in L^2(\pi): \mathbb{E}[f]=0}\frac{\|\mathsf{P}f\|_{2,\pi}}{\|f\|_{2,\pi}}.
    \end{equation*}
\end{defn}

It is well-known that the spectral gap governs the convergence stationarity. We will use the following result:

\begin{lem}
\label{lem:mixing}
    Suppose that $\pi$ is stationary for $\mathsf{P}$ and that $\mathsf{P}$ has spectral gap at least $\gamma$. Then for any distribution $\mu$ and any $t\in \mathbb{N}$, it holds that
    \begin{equation*}
        \|\pi-\mu\mathsf{P}^t\|_1\leq (1-\gamma)^t\sqrt{2\|\mu/\pi\|_{\infty}}.
    \end{equation*}
\end{lem}
\begin{proof}
    From the proof of Proposition 1.2. of~\cite{montenegro2006mathematical}, it holds from the Cauchy-Schwarz inequality that
    \begin{align*}
        \|\pi-\mu \mathsf{P}^t\|_1&=\left\|\bm{1}-\frac{\mu\mathsf{P}^t}{\pi}\right\|_{1,\pi}\\
        &\leq \left\|\bm{1}-\frac{\mu\mathsf{P}^t}{\pi}\right\|_{2,\pi}\\
        &\leq (1-\gamma(\mathsf{P}^*))^t \left\|1-\frac{\mu}{\pi}\right\|_{2,\pi}\\
        &\leq (1-\gamma(\mathsf{P}^*))^t \sqrt{\left(\sum_{x}\pi(x)\left\vert \frac{\mu}{\pi}(x)-1\right\vert\right)\max_{x}\left\vert \frac{\mu}{\pi}(x)-1\right\vert}\\    
        &\leq (1-\gamma(\mathsf{P}^*))^t\sqrt{2\|\mu/\pi\|_{\infty}},
    \end{align*}
    where in the last line we use the fact that the $\ell_1$ distance between two distributions is at most $2$. It is well-known that if $\pi$ is stationary for $\mathsf{P}$, then $\gamma(\mathsf{P}^*)=\gamma(\mathsf{P})$, see e.g. Remark 1.19 of Montenegro-Tetali.
\end{proof}

\section{Algorithmic Recovery via PageRank}

In this section, we turn to our main algorithmic results. We first require the following simple bound, which appears essentially as Lemma 1 of~\cite{subramanian}, and quantifies the basic tension between fast mixing and (lack of) approximate stationarity of a distribution with respect to a transition matrix:
\begin{lem}
\label{lem:coupling_main}
    Let $\mathsf{P}$ be any Markov chain with a stationary distribution $\pi$. Then for any distribution $\nu$ and any $t\in \mathbb{N}$ it holds that
    \begin{equation*}
        \|\pi-\nu\|_1\leq \|\pi -\nu\mathsf{P}^t\|_1+t\|\nu-\nu\mathsf{P}\|_1.
    \end{equation*}
\end{lem}
\begin{proof}
    By the triangle inequality,
    \begin{equation*}
        \|\pi-\nu\|_1\leq \|\pi -\nu\mathsf{P}^t\|_1+\|\nu-\nu\mathsf{P}^t\|_1,
    \end{equation*}
    and hence it suffices to bound the last term. To do so, we simply telescope:
    \begin{equation*}
\|\nu-\nu\mathsf{P}^t\|_1 \leq \sum_{\ell=1}^t \|\nu\mathsf{P}^{\ell-1}-\nu \mathsf{P}^{\ell}\|_1=\sum_{\ell=1}^t \|(\nu-\nu \mathsf{P})\mathsf{P}^{\ell-1}\|_1. 
    \end{equation*}
    Since $\mathsf{P}$ is a Markov kernel, it is a contraction in total variation distance and hence each term in the sum can be bounded by the first term. This proves the claim.
\end{proof}

Next, we use this result to provide a fairly well-known bound on the distance between an uncorrupted stationary distribution and the PageRank version; this result essentially appears as Theorem 1(a) of ~\cite{subramanian} as well as Theorem 2.1 of~\cite{mitrophanov2003stability}, just more explicitly quantifying the dependence on the restart distribution relative to $\pi$.
\begin{cor}
\label{cor:pr_close}
Suppose that a Markov chain $\mathsf{P}$ has spectral gap at least $\gamma$ with stationary distribution $\pi$. For any $\delta\in [0,1]$ and distribution $\mu$, consider the PageRank chain $\mathsf{P}(\delta) = (1-\delta)\mathsf{P}+\delta\bm{1}^T\mu$ and let $\pi_{\delta}$ denote the stationary distribution. Then 
\begin{equation*}
    \|\pi-\pi_{\delta}\|_1\leq \frac{C\delta\log\left(\frac{\gamma\|\mu/\pi\|_{\infty}}{\delta}\right)}{\gamma},
\end{equation*}
for an absolute constant $C>0$.
\end{cor}
\begin{proof}
    By applying \Cref{lem:coupling_main} and then \Cref{lem:mixing} for the first term, we may first write
    \begin{equation*}
        \|\pi-\pi_{\delta}\|_1\leq \sqrt{2\|\pi_{\delta}/\pi\|_{\infty}}\cdot(1-\gamma)^{t}+t\|\pi_{\delta}-\pi_{\delta}\mathsf{P}\|_1\leq \sqrt{2\|\pi_{\delta}/\pi\|_{\infty}}\exp(-t\gamma)+t\|\pi_{\delta}-\pi_{\delta}\mathsf{P}\|_1.
    \end{equation*}

    For the second term, the stationarity of $\pi_{\delta}$ with respect to $\mathsf{P}(\delta)$ implies that
    \begin{align*}
        \|\pi_{\delta}-\pi_{\delta}\mathsf{P}\|_1&=\|\pi_{\delta}\mathsf{P}(\delta)-\pi_{\delta}\mathsf{P}\|_1\\
        &=\delta\|\pi_{\delta}(\mathsf{P}-\bm{1}^T\mu)\|_1\\
        &\leq 2\delta.
    \end{align*}

    Setting $t=\frac{C\log\left(\gamma\frac{\|\pi_{\delta}/\pi\|_{\infty}}{\delta}\right)}{\gamma}$ for a large enough constant $C>0$ yields the bound
    \begin{equation*}
        \|\pi-\pi_{\delta}\|_1\leq  \frac{C'\delta\log\left(\frac{\gamma\|\pi_{\delta}/\pi\|_{\infty}}{\delta}\right)}{\gamma},
    \end{equation*}
    for some slightly larger constant $C'$. To replace the final bound to be in terms of $\mu$, we can easily argue as follows: let $f=(\frac{\mu}{\pi})^T$ be the relative density as a column vector. Then by \Cref{fact:adj_dist}, 
    \begin{equation*}
        (\mathsf{P}^*)^kf = \left(\frac{\mu \mathsf{P}^k}{\pi}\right)^T.
    \end{equation*}
    Since $\mathsf{P}^*$ is also a Markov kernel, it is a contraction in $\ell_{\infty}$ by \Cref{fact:contract}, and hence 
    \begin{equation*}
        \sup_k \left\|\frac{\mu \mathsf{P}^k}{\pi}\right\|_{\infty}=\sup_k \|(\mathsf{P}^*)^kf\|_{\infty} = \|f\|_{\infty}.
    \end{equation*}
    On the other hand, observe that $\mathsf{P}\bm{1}^T=\bm{1}^T$ since $\mathsf{P}$ is row-stochastic. As a result, since $\pi_{\delta}$ satisfies the equation
    \begin{equation*}
        \pi_{\delta} = \pi_{\delta}(\delta \bm{1}^T\mu+(1-\delta)\mathsf{P})=\delta\mu+(1-\delta)\pi_{\delta}\mathsf{P},
    \end{equation*}
    or equivalently $\pi_{\delta} = \delta\mu(\mathsf{I}-(1-\delta)\mathsf{P})^{-1}$, the von Neumann expansion yields
    \begin{equation*}
        \pi_{\delta} = \delta\sum_{t=0}^{\infty}(1-\delta)^t\mu\mathsf{P}^t.
    \end{equation*}
    Since this is a convex combination, it follows that
    \begin{equation}
    \label{eq:contract}
        \left\|\frac{\pi_{\delta}}{\pi}\right\|_{\infty}\leq \sup_k \left\|\frac{\mu \mathsf{P}^k}{\pi}\right\|_{\infty}=\left\|\frac{\mu}{\pi}\right\|_{\infty}
    \end{equation}
    as claimed.
\end{proof}

We can now provide our main algorithmic observation: we can employ \Cref{lem:coupling_main} \emph{once again} with respect to the corrupted PageRank matrix. Indeed, \emph{any} PageRank matrix inherits favorable spectral properties that improve with $\delta$:
\begin{cor}
\label{cor:corrupted_close}
    Let $\mathsf{P}$ be a Markov chain with stationary distribution $\pi$ and suppose that $\widetilde{\mathsf{P}}$ is a $\varepsilon$-corruption. For any $\delta>0$ and distribution $\mu$, let $\mathsf{P}(\delta)$ and $\widetilde{\mathsf{P}}(\delta)$ denote the PageRank version with restart $\mu$ with stationary distribution $\pi_{\delta}$ and $\widetilde{\pi}_{\delta}$ as before. Then if $\mu$  is $(\beta,p)$-smooth with respect to $\pi$, 

    \begin{equation*}
        \|\pi_{\delta}-\widetilde{\pi}_{\delta}\|_1\leq \frac{C\varepsilon^{1/q}\log(1/\varepsilon)\beta}{\delta}. 
    \end{equation*}
    for an absolute constant $C>0$, where $q$ is the dual exponent of $p$.

\end{cor}
\begin{proof}
    We again apply \Cref{lem:coupling_main}, but now replacing the role of $\mathsf{P}$ with $\widetilde{\mathsf{P}}_{\delta}$:
    \begin{align*}
        \|\pi_{\delta}-\widetilde{\pi}_{\delta}\|_1&\leq \| \widetilde{\pi_{\delta}}-\pi_{\delta}\widetilde{\mathsf{P}}(\delta)^t\|_1+t\|\pi_{\delta}-\pi_{\delta}\widetilde{\mathsf{P}}(\delta)\|_1.
    \end{align*}
    For the first term, we can use a standard coupling argument as in [Corollary 5.5., Levin-Peres] to assert that 
    \begin{equation*}
        \| \widetilde{\pi_{\delta}}-\pi_{\delta}\widetilde{\mathsf{P}}(\delta)^t\|_1\leq 2\Pr(\tau_{\delta}\geq t)\leq 2\exp(-\delta t),
    \end{equation*}
    where $\tau_{\delta}$ is a geometric random variable with success probability $\delta$ since the total variation distance is bounded by the time that it takes to choose to randomly restart both chains in an optimal coupling starting with $\widetilde{\pi}_{\delta}$ and $\pi_{\delta}$ when transitioning via $\widetilde{\mathsf{P}}$. For the second term, stationarity of $\pi_{\delta}$ with respect to $\mathsf{P}(\delta)$ implies:
    \begin{align*}
        \|\pi_{\delta}-\pi_{\delta}\widetilde{\mathsf{P}}(\delta)\|_1&=\|\pi_{\delta}\mathsf{P}(\delta)-\pi_{\delta}\widetilde{\mathsf{P}}(\delta)\|_1\\
        &=(1-\delta)\|\pi_{\delta}(\mathsf{P}-\widetilde{\mathsf{P}})\|_1\\
        &\leq \sum_{x,y}\pi_{\delta}(x)\vert \mathsf{P}(x,y)-\widetilde{\mathsf{P}}(x,y)\vert.
    \end{align*}
    To conclude, suppose that $\mu$ is $(\beta,p)$-smooth with respect to $\pi$. In this case, using H\"older's inequality gives
    \begin{align*}
        \sum_{x,y}\pi_{\delta}(x)\vert \mathsf{P}(x,y)-\widetilde{\mathsf{P}}(x,y)\vert&=2\mathbb{E}_{X\sim \pi_{\delta}}[d_{\mathsf{TV}}(\mathsf{P}(X,\cdot),\widetilde{\mathsf{P}}(X,\cdot))]\\
        &=2\mathbb{E}_{X\sim \pi}\left[\frac{\mathrm{d}\pi_{\delta}}{\mathrm{d}\pi}(X)d_{\mathsf{TV}}(\mathsf{P}(X,\cdot),\widetilde{\mathsf{P}}(X,\cdot))\right]\\
        &\leq 2\left\|\frac{\pi_{\delta}}{\pi}\right\|_{p,\pi}\|d_{\mathsf{TV}}(\mathsf{P}(X,\cdot),\widetilde{\mathsf{P}}(X,\cdot))\|_{q,\pi}\\
        &\leq 2\varepsilon^{1/q}\left\|\frac{\pi_{\delta}}{\pi}\right\|_{p,\pi}
    \end{align*}
    where in the last step we use the general definition of an $\varepsilon$-corruption with respect to $\pi$ to see that
    \begin{equation*}
        \|d_{\mathsf{TV}}(\mathsf{P}(X,\cdot),\widetilde{\mathsf{P}}(X,\cdot))\|_{q,\pi}=\left(\mathbb{E}_{\pi}[d^q_{\mathsf{TV}}(\mathsf{P}(X,\cdot),\widetilde{\mathsf{P}}(X,\cdot))]\right)^{1/q}\leq \varepsilon^{1/q},
    \end{equation*}
    since the total variation is at most $1$ surely.
    But by the same argument as in \Cref{cor:pr_close} using the contractivity of $\mathsf{P}^*$ in $L^p(\pi)$ by \Cref{fact:contract}, it follows that
    \begin{equation*}
        \left\|\frac{\pi_{\delta}}{\pi}\right\|_{p,\pi}\leq \left\|\frac{\mu}{\pi}\right\|_{p,\pi}\leq \beta,
    \end{equation*}
    by assumption. Putting this together, we find that
    \begin{equation*}
        \|\pi_{\delta}-\widetilde{\pi}_{\delta}\|_1\leq 2\exp(-\delta t)+2\varepsilon^{1/q} \beta t .
    \end{equation*}
    
    We now set $t=\frac{\log(1/\varepsilon)}{\delta}$ to deduce that 
    \begin{equation*}
        \|\pi_{\delta}-\widetilde{\pi}_{\delta}\|_1\leq \frac{C\varepsilon^{1/q}\log(1/\varepsilon)\beta}{\delta}.
    \end{equation*}

\end{proof}

With these preliminary estimates in order, the final observation is that we can balance the scaling of the PageRank matrix with a $(\beta,p)$-smooth restart in a way that scales with both the corruption level and original spectral gap, even though the corrupted chain itself will crucially \emph{not} have any favorable spectral properties. In other words, the PageRank construction serves as a convenient regularizer that controls the effects of the general perturbations allowed under \Cref{def:corruption} under mild conditions on the restart distribution.
\begin{thm}
\label{thm:main_body}
        Let $\mathsf{P}$ be a Markov chain with stationary distribution $\pi$ and spectral gap at least $\gamma$, and suppose that $\widetilde{\mathsf{P}}$ is a $\varepsilon$-corruption. Given the corrupted chain $\widetilde{\mathsf{P}}$ and a restart distribution $\mu$ that is $\beta$-smooth with respect to $\pi$, there is an algorithm that outputs a distribution $\widehat{\pi}$ such that
    \begin{equation*}
        d_{\mathsf{TV}}(\widehat{\pi},\pi)=O\left(\sqrt{\frac{\beta\varepsilon^{1/q}\log\left(\frac{\beta \|\mu/\pi\|_{\infty}}{\varepsilon\gamma}\right)}{\gamma}}\right).
    \end{equation*}
\end{thm}
\begin{proof}
We simply take $\widehat{\pi}$ to be $\widetilde{\pi}_{\delta}$ for suitably chosen $\delta$ for the PageRank chain induced by $\widetilde{\mathsf{P}}$ and $\mu$. By \Cref{cor:pr_close} and \Cref{cor:corrupted_close} and the triangle inequality,
\begin{align*}
    \|\widehat{\pi}-\pi\|_1&\leq \|\widetilde{\pi}_{\delta}-\pi_{\delta}\|_1+\|\pi_{\delta}-\pi\|_1\\
    &\leq C\left(\frac{\varepsilon^{1/q}\log(1/\varepsilon)\beta}{\delta}+\frac{\delta\log\left(\frac{\|\mu/\pi\|_{\infty}}{\delta}\right)}{\gamma}\right).
\end{align*}
Setting 
\begin{equation*}
    \delta=\Theta\left(\sqrt{\frac{\gamma\varepsilon^{1/q}\log(1/\varepsilon)\beta}{\log(\|\mu/\pi\|_{\infty})}}\right)
\end{equation*}
to approximately balance terms yields
\begin{equation*}
    d_{\mathsf{TV}}(\widehat{\pi},\pi)=O\left(\sqrt{\frac{\beta\varepsilon^{1/q}\log\left(\frac{\beta \|\mu/\pi\|_{\infty}}{\varepsilon\gamma}\right)}{\gamma}}\right).
\end{equation*}
\end{proof}

As an illustrative example, a simple corollary of this result is that in the setting where it is known that that the original stationary distribution is suitably \emph{spread}, one can immediately obtain an approximate version of the stationary distribution under an arbitrary $\varepsilon$-node corruption of the Markov chain:

\begin{cor}
\label{cor:spread}
    Let $\mathsf{P}$ be a Markov chain with stationary distribution $\pi$ and spectral gap at least $\gamma$. Suppose further that it is known that $\alpha/n\leq \pi(x)\le (1/\alpha) n$ for some $\alpha>0$ and that $\widetilde{\mathsf{P}}$ differs arbitrarily from $\mathsf{P}$ on an $\varepsilon$-fraction of rows. Then there is an algorithm that, given the corrupted chain $\widetilde{\mathsf{P}}$, outputs a distribution $\widehat{\pi}$ such that
    \begin{equation*}
        d_{\mathsf{TV}}(\widehat{\pi},\pi)=\widetilde{O}\left(\frac{1}{\alpha}\sqrt{\frac{\varepsilon}{\gamma}}\right).
    \end{equation*}
\end{cor}
\begin{proof}
    We may simply instantiate \Cref{thm:main_body} using the uniform restart distribution $\mu=\bm{1}/n$, noting that this distribution is $(\beta,\infty)$-smooth for $\beta=1/\alpha$ and that perturbing a $\varepsilon$-fraction of rows under the assumption is a $\varepsilon/\alpha$-corruption .
\end{proof}

We conclude with some remarks about \Cref{def:corruption} and \Cref{thm:main_body}.
\begin{rmk}[Measuring Corruptions via $\pi$]
\label{rmk:pi_neccessary}
    It is not too difficult to see that for general Markov chains, the quantification of corruption must account for the heterogeneous importance of sites as measured by the stationary distribution. For instance, consider the star graph on $n+1$ vertices where each site holds with probability $1/2$, and outer nodes move to the center with probability $1/2$, while the central node moves with probability $1/2n$ to each outer node. The stationary distribution $\pi_1$ is easily seen to have mass $1/2$ on the center and $1/2n$ for each outer node by inspection. It is straightforward to see that the mixing time is $\Theta(1)$, since the chain perfectly mixes after transitioning from the center node.

    However, consider the corruption of just two nodes where the center and a single outer node are changed as follows: the central node instead transitions to a distinguished outer node with probability $1/4$ and scales down the other transitions to $1/4(n-1)$. It is straightforward to compute that the stationary distribution $\widetilde{\pi}$ places mass $1/2$ on the central node, $1/4$ on the distinguished outer node, and $1/4(n-1)$ on the remaining outer nodes; moreover, the mixing time is again $\Theta(1)$ since the chain again perfectly mixes once at the central node. In particular, both the original chain and corrupted chains satisfy similar lower bounds on the probability mass of all nodes and have similar spectral properties; therefore, it is information-theoretically impossible to tell which was the original chain. Since $\mathsf{d}_{\mathsf{TV}}(\pi,\widetilde{\pi})\approx 1/2$, no estimator can recover the stationary distribution to accuracy better than roughly $1/4$. The issue is that the central node can exert a large influence on the stationary distribution without affecting much of the spectral properties; therefore, it is necessary for the notion of corruption to reflect this relative importance even for otherwise nice chains.
\end{rmk}

\begin{rmk}
\label{rmk:gap}
    Similarly, it is easy to see that the behavior of $\varepsilon/\gamma$ governs the information-theoretic recoverability of the stationary distribution. In particular, it is impossible to recover the stationary distribution to accuracy better than $\varepsilon/2\gamma$. To see this, consider the product chain on $\{-1,1\}^n$ where at each step, a random coordinate in $[n]$ resamples the $i$th coordinate to be $1$ with probability $p_i$ and $-1$ with probability $1-p_i$. For any choice of the $p_i$, this distribution is a product chain and hence has spectral gap $\gamma=1/n$. On the other hand, the two chains where $p_1=1/2+\delta/2$ and $\widetilde{p}_1=1/2-\delta/2$ can be viewed as $\varepsilon=\delta/n$ corruptions of each other, and the corresponding total variation distance is also $\delta=\varepsilon/n$ simply by considering the first coordinate. Therefore, no algorithm can guess which was the original chain better than $\varepsilon/2\gamma$ in total variation.
\end{rmk}

\section{Conclusion and Open Questions}
In this work, we have shown how under mild conditions, one can approximately recover the true stationary distribution under general perturbations beyond what can be attained using matrix perturbation theory. There are several immediate questions that are exciting directions for future work: 
\begin{enumerate}
    \item As we discussed previously, the  asymptotic behavior of $\varepsilon/\gamma$ determines whether recovery is attainable even under the restricted family of perturbations considered in prior work. As we discussed, under the weaker family of corruptions obtained by $\varepsilon$ perturbations of each site, several authors have used variants of \Cref{lem:coupling_main} to attain a linear dependence on this quantity just by computing the new, corrupted distribution. 
    Is near-linear scaling possible for the general corruptions considered under \Cref{def:corruption} even in the setting of \Cref{thm:main_body} with $p=\infty$?
    \item To what extent can one relax the assumption that there is a known restart $\mu$ that is smooth? Can this assumption be relaxed to hold just on small enough subsets or dispensed with altogether? As we argued in our motivating example, it may be reasonable to assume approximate knowledge of the uncorrupted chain. But one hint that such an assumption may not be algorithmically necessary is that if there are states with atypically small mass relative to $\gamma$, then the larger states of the uncorrupted Markov chain should not transition there much. Under \Cref{def:corruption}, this should remain approximately remain true for the corrupted transitions, which suggests that one may be able to use the corrupted transitions to identify atypically small mass states in the true stationary distribution. The challenge, of course, is that the large states will need to be determined from the corrupted chain.
    \item The approach taken by this work is to assume a spectral gap bound and an $\ell_p$ bounded estimate of the original chain. Are there interesting, restricted families of Markov chains and perturbations where recovery is possible under alternate structural assumptions, and if so, what algorithms attain them? 
    \item In several applications, like high-dimensional spin systems, recovering the true stationary distribution up to total variation may be challenging due to the exponential state space in the number of physical sites. Rather, one may instead wish to compute a distribution that approximates the true stationary distribution with respect to some particular class of test functions $\mathcal{F}$. The total variation metric corresponds to the very large class of all bounded functions. But in important applications, one may instead care about simpler functions, like low-degree polynomials, which is still broad enough to capture physically important quantities like the magnetization. Can improved guarantees be obtained under weaker conditions for these kinds of interesting settings?
\end{enumerate}

\bibliographystyle{alpha}
\bibliography{bibliography}

\end{document}